\documentclass[12pt,reqno]{amsart}
\usepackage[top=2cm,bottom=2cm,right=2.5cm,left=2.5cm]{geometry}
\usepackage{amssymb}
\usepackage{amsmath, amsthm, amscd, amsfonts, amssymb, graphicx, color}
\usepackage[bookmarksnumbered, colorlinks, plainpages]{hyperref}
 
\usepackage{hyperref}

\newtheorem{theorem}{Theorem}[section]

\newtheorem{corollary}[theorem]{Corollary}
\newtheorem{proposition}[theorem]{Proposition}
\theoremstyle{definition}
\newtheorem{definition}[theorem]{Definition}

\theoremstyle{remark}
\newtheorem{remark}[theorem]{Remark}

\numberwithin{equation}{section}
\usepackage[colorlinks]{hyperref}
\hypersetup{colorlinks=true,linkcolor=red, anchorcolor=blue,citecolor=blue, urlcolor=red, filecolor=magenta, pdftoolbar=true}

\begin{document}
	
\setcounter{page}{1}

\title[On the stability of radical functional equation in modular space]{On the stability of radical functional equation in modular space}
\author[A. Baza, M. Rossafi]{Abderrahman Baza$^{1}$ and Mohamed Rossafi$^{2*}$}

\address{$^{1}$Laboratory of Analysis, Geometry and Application, Departement of Mathematics,  Ibn Tofail University, Kenitra,  Morocco}
\email{\textcolor[rgb]{0.00,0.00,0.84}{abderrahmane.baza@gmail.com}}

\address{$^{2}$Laboratory Partial Differential Equations, Spectral Algebra and Geometry, Higher School of Education and Training, Ibn Tofail University, Kenitra, Morocco}
\email{\textcolor[rgb]{0.00,0.00,0.84}{mohamed.rossafi1@uit.ac.ma}}

\subjclass[2020]{Primary 39B82, Secondary 39B52}

\keywords{modular spaces, Hyers-Ulam stability, Fatou property, $\Delta_2$-condition}

\date{
	\newline \indent $^{*}$Corresponding author}

\maketitle

\begin{abstract}
  In this work, we prove the generalised Hyer Ulam stability of the following functional equation
  \begin{equation}\label{Eq-1}
  	\phi(x)+\phi(y)+\phi(z)=q \phi\left(\sqrt[s]{\frac{x^s+y^s+z^s}{q}}\right),\qquad |q| \leq 1
  \end{equation} and $s$ is an odd integer such that $s\geq 3$,
  in modular space, using the direct method, and the fixed point theorem.
\end{abstract}
 
 	\baselineskip=12.4pt
 
\section{Introduction and preliminaries}
The modulars and modular spaces as generalisation of normed space were by Nakano \cite{Nakano}. 
Since the 1950s, a number of eminent mathematicians \cite{Amemiya, Koshi, Musielak, Orlicz,Yamamuro} have focused extensively on it.
 Two applications of modular and modular spaces are found in \cite{Maligranda, Krbec, Orlicz, Yamamuro} interpolation theory and Orlicz spaces.
 
We begin by introducing the definition and some properties of modular and modular spaces.
\begin{definition}\label{Definition1.1}
	Let $Y$ be an arbitrary vector space. A functional $\rho: Y \rightarrow[0, \infty)$ is called a modular if for arbitrary $u, v \in Y$:
	\begin{enumerate}
		\item $\rho(u)=0$ if and only if $u=0$.\label{item1}
		\item $\rho(\alpha u)=\rho(u)$ for every scalar $\alpha$ with $|\alpha|=1$.
		\item $\rho(\alpha u+\beta v) \leq \rho(u)+\rho(v)$ if and only if $\alpha+\beta=1$ and $\alpha ,\beta \geq 0$.
		
		If \eqref{item1} is replaced by:
		\item $\rho(\alpha u+\beta v) \leq \alpha \rho(u)+\beta \rho(v)$ if and only if $\alpha+\beta=1$ and $\alpha, \beta \geq 0$, then we say that $\rho$ is a convex modular.
		\item 
		It is said that the modular $\rho$ has the Fatou property if and only if $\rho(x) \leq \liminf_{n\to\infty}\rho(x_n)$ whenever $x=\rho-\lim_{n\to\infty}x_n$.
	\end{enumerate}
	A modular $\rho$ defines a corresponding modular space, i.e., the vector space $Y_\rho$ given by:
	\begin{equation*}
		Y_\rho=\{u \in Y: \rho(\lambda u) \rightarrow 0 \text { as } \lambda \rightarrow 0\}.
	\end{equation*}
	A function modular is said to satisfy the $\Delta_k$-condition if there exist $\tau_{k} > 0$ such that $\rho(ku) \leq \tau_{k} \rho(u)$ for all $u \in Y_\rho$.
\end{definition}
\begin{definition}
	Let $\left\{u_n\right\}$ and $u$ be in $Y_\rho$. Then:
	\begin{enumerate}
		\item 
		The sequence $\left\{u_n\right\}$, with $u_n \in Y_\rho$, is $\rho$-convergent to $u$ and write: $u_n \to  u$ if $\rho\left(u_n-u\right)\to 0$ as $n \rightarrow \infty$.
		\item
		The sequence $\{u_n\}$, with $u_n \in Y_\rho$, is called $\rho$-Cauchy if $\rho\left(u_n-u_m\right)\to 0$ as $n, m \to \infty$.
		\item
		$Y_\rho$ is called $\rho$-complete if every $\rho$-Cauchy sequence in $Y_\rho$ is $\rho$-convergent.
	\end{enumerate}
\end{definition}
\begin{proposition}
	In modular space,
	\begin{itemize}
		\item If $u_n \overset{\rho}{\to} u$ and a is a constant vector, then $u_n+a \overset{\rho}{\to} u+a$.
		\item If $u_n \overset{\rho}{\to} u$ and $v_n \overset{\rho}{\to} v$ then $\alpha u_n + \beta v_n \overset{\rho}{\to} \alpha u+ \beta v$, where $\alpha+\beta \leq 1$ and $\alpha,\beta \geq 1$. 
	\end{itemize}
\end{proposition}
\begin{remark}
	Note that $\rho(u)$ is an increasing function, for all $u \in X$. 
	Suppose $0<a<b$, then property $(4)$ of Definition \ref{Definition1.1} with $v=0$ shows that $\rho(a u)=\rho\left(\dfrac{a}{b} b u\right) \leq \rho(b u)$ for all $u \in Y$. 
	Morever, if $\rho$ is a convexe modular on $Y$ and $|\alpha| \leq 1$, then $\rho(\alpha u) \leq \alpha \rho(u)$.
	
	In general, if $\lambda_i \geq 0$, $i=1, \dots,n$ and  $\sum \lambda_i,= 1$ then $\rho (\lambda_1 u_1+\lambda_2 u_2+\dots+\lambda_n u_n) \leq \lambda_1 \rho(u_1)+\lambda_2 \rho(u_2)+\dots+\lambda_n \rho(u_n)$.
	
	If $\{u_n\}$ is $\rho$-convergent to $u$, then $\{ c u_n \}$ is $\rho$-convergent to $cu$, where $|c| \leq 1$.
	But the $\rho$-convergent of a sequence $\{u_n\}$ to $u$ does not imply that $\{\alpha u_n\}$ is $\rho$-convergent to $\alpha u_n$ for scalars $\alpha $ with $|\alpha|>1$.
	
	If $\rho$ is a convex modular satisfying $\Delta_k$ condition with $0 <\tau_{k}<k$, then $\rho(u) \leq \tau_{k} \rho(\dfrac{1}{k} u) \leq \dfrac{\tau_{k}}{k} \rho(u)$ for all $u$.
	
	Hence $\rho=0$. Consequently, we must have $\tau_{k} \geq k$ if $\rho$ is convex modular. 
\end{remark}
In many settings, the functional equations are essential for studying stability issues. 

The work on stability concerns began with Ulam's initial doubts about the stability of group homomorphisms \cite{Ulam}. Hyers \cite{Hyers} used Banach spaces to tackle this stability issue by taking in to account Cauchy's functional equation. 

Aoki \cite{Aoki} extented the work of Hyers by assuming an infinite Cauchy difference. The additive mapping work was given by Rassias \cite{Rassias}, and Gavrüta \cite{Găvruţa} presented similar results.

In this work, we study the stability of the following functional equation:
\begin{equation*}
	f(x)+f(y)+f(z)=q f\left(\sqrt[s]{\frac{x^s+y^s+z^s}{q}}\right) ;\qquad |q| \leq 1
\end{equation*}
in modular space, using the direct method and fixed point theorem.
\begin{proposition}
	Let $\phi: \mathbb{R} \rightarrow Y$ be a mapping satisfying \eqref{Eq-1}, where $Y$ is a vector space. Then $\phi$ confirm: $\phi(\sqrt[s]{x^s+y^s})=\phi(x)+\phi(y), \forall x, y \in \mathbb{R}$.
\end{proposition}
\begin{proof}
	Letting $x=y=z=0$ in \eqref{Eq-1}, we get: $3 \phi(0)=q \phi(0)$. \\
	Hence $\phi(0)=0$.
	Letting $z=0$ and $y=-x$ in \eqref{Eq-1}, we obtain:
	\begin{equation*}
		\phi(x)+\phi(-x)=0, 
	\end{equation*}
	Hence
	\begin{equation*}
		\phi(-x)=-\phi(x)
	\end{equation*}
	Letting $y=x$ and $z=-\sqrt[s]{x^s+y^s}$, we get
	$$
		\phi(x)+\phi(y)-\phi\left(\sqrt[s]{x^s+y^s}\right)=0,$$   then $$\phi\left(\sqrt[s]{x^s+y^s}\right)=\phi(x)+\phi(y)
	$$
\end{proof}
\section{Stability of radical function equation \eqref{Eq-1} in modular space satisfying $\Delta_2$-condition.}
\begin{theorem}\label{Theorem-1}
	Let $Y_\rho$ be a $\rho$-complete convex modular space satisfying $\Delta_2$-condition with$\tau$ and $\phi: \mathbb{R} \longrightarrow Y_\rho$ an old mapping such that:
	\begin{equation}\label{Eq-2}
		\rho(\phi(x)+\phi(y)+\phi(z)-q \phi\left(\sqrt[s]{\frac{x^s+y^s+z^s}{q}}\right)) \leq  \alpha(x, y, z)  
	\end{equation}
	for all $x, y, z \in \mathbb{R}$ and all integer $q$ with $|q| \leq  1$, and $\alpha: \mathbb{R}^3 \rightarrow[0, \infty)$ is a mapping such that:
	\begin{equation} \label{Eq-3}
		\lim_{n \rightarrow \infty} \tau^n \alpha\left(\frac{x}{2^{n / s}},\frac{y}{2^{n / s}}, \frac{z}{2^{n / s}}\right)=0 
		\text{ and }
		\sum_{j=1}^{\infty}\left(\frac{\tau^2}{2}\right)^i \alpha\left(\frac{x}{2^{j / s}} , \frac{x}{2^{j / s}}, \frac{-x}{2^{\frac{j-1}{s}}}\right) < \infty
	\end{equation}
	for all $x, y, z \in \mathbb{R}$. 
	Then there exists a unique radical mapping $A: \mathbb{R} \rightarrow Y_\rho$   satisfying \eqref{Eq-1} and 
	\begin{equation} \label{Eq-4}
		\rho(\phi(x)-A(x)) \leq  \frac{1}{2} \sum_{j=1}^{\infty}\left(\frac{\tau^2}{2}\right)^{j} \alpha\left(\frac{x}{2^{j / s}}, \frac{x}{2^{j/ s}}, \frac{-x}{2^{\frac{j-1}{s}}}\right), x\in \mathbb{R}
	\end{equation}
\end{theorem}
\begin{proof}
	First, we have $\phi(0)=0$ because $\alpha(0,0,0)=0$ by the convergence of $\sum_{i=1}^{\infty}\left(\frac{\tau^2}{2}\right)^i \alpha(0,0,0)<\infty$.
	
	 Putting $y=x$ and $z=-2^{1 / s} x$ in \eqref{Eq-2}, we get:
	\begin{equation*}
			\rho\left(2 \phi(x)-\phi\left(2^{1 / s}x\right)\right) \leq  \alpha\left(x, x,-2^{1 / s} x\right)
	\end{equation*}
	Then
	\begin{equation*}
			\rho\left(2 \phi\left(\frac{x}{2^{1 / s}}\right)-\phi(x)\right) \leq  \alpha\left(\frac{x}{2^{1 / s}}, \frac{x}{2^{1 / s}},-x\right)
	\end{equation*}
	Now, we have:
	\begin{align*}
			\rho\left(\phi(x)-2^n \phi\left(\frac{x}{2^{n/s}}\right)\right)
			&=\rho\left(\sum_{j=1}^n \frac{1}{2^{j}}\left(2^{2 j-1} \phi\left(\frac{x}{2^{\frac{j-1}{s}}}\right)-2^{2 j} \phi\left(\frac{x}{2^{j / s}}\right)\right)\right) \\
			& \leq  \frac{1}{\tau} \sum_{j=1}^n\left(\frac{\tau^2}{2}\right)^j \times\left(\frac{x}{2^{j / s}}, \frac{x}{2^{j / s}},-\frac{x}{2^{\frac{j-1}{s}}}\right)
	\end{align*}
	Let $m, n$ be positives integers such that $n>m$, we have:
	\begin{align*}
		\rho\left(2^n \phi\left(\frac{x}{2^{n / s}}\right)-2^m \phi\left(\frac{x}{2^{m / s}}\right)\right)
		& \leq  \tau^m \rho\left(2^{n-m} \phi\left(\frac{x}{2^{\frac{n-m}{s}} \cdot 2^{\frac{m}{s}}}\right)-\phi\left(\frac{x}{2^{\frac{m}{s}}}\right)\right)\\		
		& \leq  \tau^{m-1} \sum_{j=1}^{n-m}\left(\frac{\tau^2}{2}\right)^j \alpha\left(\frac{x}{2^{\frac{m+j}{s}}}, \frac{x}{2^{\frac{m+j}{s}}},-\frac{x}{2^{\frac{m+j-1}{s}}}\right) \\
		& =\frac{1}{\tau}\left(\frac{2}{\tau}\right)^m \sum_{j=m+1}^n\left(\frac{\tau^2}{2}\right)^j \alpha\left(\frac{x}{2^{j / s}}, \frac{x}{2^{j / s}},-\frac{x}{2^{\frac{j-1}{s}}}\right)\\
		&\to 0 \text{ as } m\to \infty
	\end{align*}
	for all $x \in \mathbb{R}$.
	
	Then $\left\{2^n \phi\left(\frac{x}{2^{n / s}}\right)\right\}$ is $\rho$-Cauchy sequence in $Y_\rho$ which is $\rho$-complete. 
	So, we can define a mapping $A: \mathbb{R} \rightarrow Y_\rho$ as: $A(x)=\rho-\lim_{n \rightarrow \infty} 2^n \phi\left(\frac{x}{2^{n / s}}\right)$, i.e. $\lim_{n \rightarrow \infty} \rho\left(2^n \phi\left(\frac{x}{2^{n / s}}\right)-A(x)\right)=0$. 
	Now, we prove the estimation \eqref{Eq-4}, we have:
	\begin{align*}
			\rho(\phi(x)-A(x))&\leq\frac{1}{2} \rho\left(2 \phi(x)-2^{n+1} \phi\left(\frac{x}{2^{n / s}}\right)\right)+\frac{1}{2} \rho\left(2^{n+1} \phi\left(\frac{x}{2^{n/s}}\right)-2 A(x)\right)\\
			& \leq  \frac{\tau}{2} \rho\left(\phi(x)-2^n \phi\left(\frac{x}{2^{n / s}}\right)\right)+\frac{\tau}{2} \rho\left(2^n \phi\left(\frac{x}{2^{n / s}}\right)-A(x)\right) \\
			&\leq\frac{1}{2} \sum_{j=1}^n\left(\frac{\tau^2}{2}\right)^j \alpha\left(\frac{x}{2^{j/s}}, \frac{x}{2^{j/s}}, \frac{-x}{2^{ \frac{j-1}{s}}}\right)+ \frac{\tau}{2} \rho\left(2^n \phi\left(\frac{x}{2^{n / s}}\right)-A(x)\right).
	\end{align*}
	for all $x \in \mathbb{R}$. Taking $n \rightarrow \infty$ we obtain the estimation \eqref{Eq-4}.
	
	We have:
	\begin{equation*}
			\rho\left(2^n \phi\left(\frac{x}{2^{n / s}}\right)+2^n \phi\left(\frac{y}{2^{n / s}}\right)+2^n \phi\left(\frac{3}{2^{n / s}}\right)-2^n q \phi\left(\frac{1}{2^n} \sqrt[5]{\frac{x^s+y^s+z^s}{2}}\right)\right)
			\leq  \tau^n \alpha\left(\frac{x}{2^{n / s}}, \frac{y}{2^{n / s}},-\frac{z}{2^{n / s}}\right)
	\end{equation*}
	for all $x, y, z \in \mathbb{R}$ and all positive integer $n$.
	And we have:
	\begin{multline*}
			\rho\left(\frac{1}{5} A(x)+\frac{1}{5} A(y)+\frac{1 }{5}A(z)-\frac{1}{5}q A\left(\sqrt[s]{\frac{x^s+y^s+z^s}{q}}\right)\right)
			\\
			\leq  \frac{1}{5} \rho\left(A(x)-2^n \phi\left(\frac{x}{2^{n / s}}\right)\right)+\frac{1}{5} \rho\left(A(y)-2^n \phi\left(\frac{y}{2^{n / s}}\right)\right) \\ 
			+\frac{1}{5} \rho\left(A(z)-2^n \phi\left(\frac{z}{2^{n / s}}\right)\right)+\frac{1}{5}\rho\left(q A\left(\sqrt[s]{\frac{x^s+y^s+z^s}{q}}\right)-2^nq \phi\left(\frac{1}{2^{n / s}} \sqrt[5]{\frac{x^s+y^s+z^s}{q}}\right)\right)
		\\
		+\frac{1}{5} \rho\left(2^n \phi\left(\frac{x}{2^{n/s}}\right)+2^n \phi\left(\frac{y}{2^{n / s}}\right)+2^n \phi\left(\frac{z}{2^{n / s}}\right)-q \cdot 2^n \phi\left(\frac{1}{2^{n / s}} \sqrt[5]{
		{\frac{x^s+y^s+z^s}{q}}}\right)
		\right) \\ 
		 \leq \dfrac{1}{5} \rho \left(A(x)-2^n \phi\left(\frac{x}{2^{n / s}}\right)\right)+\frac{1}{5} \rho\left(A(y)-2^n \phi\left(\frac{y}{2^{n/s}}\right)\right) 
		 +\frac{1}{5} \rho\left(A(z)-2^n \phi\left(\frac{z}{2^{n/s}}\right)\right)\\
		 + \frac{|q|}{5} \rho\left(A\left(\sqrt[5]{\frac{x^s+y^s+z^s}{q}}\right)
		 - 2^n \phi\left(\frac{1}{2^{n / s}} \sqrt[5]{\frac{x^s+y^s+z^s}{q}}\right)\right)\\
		  +\frac{1}{5} \tau^n \alpha\left(\frac{x}{2^{n / s}}, \frac{y}{2^{n / s}} ,\frac{z}{2^{n / s}}\right) \rightarrow 0 \text { as } n \rightarrow \infty.
	\end{multline*}
	Hence: 
	$$A(x)+A(y)+A(z)=q A\left(\sqrt[s]{\frac{x^s+y^s+z^s}{q}}\right) .$$
	Now, we prove the uniqueness of $A$.
	
 	Let $A' : \mathbb{R} \rightarrow Y_\rho$ be an other mapping satisfying the estimation \eqref{Eq-4}.
	
	We have:
	\begin{align*}
			\rho\left(A(x)-A^{\prime}(x)\right)   & \leq   \frac{1}{2} \rho\left(2^{n+1} A\left(\frac{x}{2^{n/s}}\right)-2^{n+1} \phi\left(\frac{x}{2^{n/s}}\right)\right) \\
			& +\frac{1}{2} \rho\left(2^{n+1} \phi\left(\frac{x}{2^{n / s}}\right)-2^{n+1} A^{\prime}\left(\frac{x}{2^{n/s}}\right)\right) \\
			\leq  & \frac{\tau^{n+1}}{2} \rho\left(A\left(\frac{x}{2^{n / s}}\right)-\phi\left(\frac{x}{2^{n / s}}\right)\right)+\frac{\tau^{n+1}}{2} \rho\left(\phi\left(\frac{x}{2^{n / s}}\right)-A^{\prime}\left(\frac{x}{2^{n / s}}\right)\right) \\
			\leq  & \frac{\tau^{n+1}}{2} \sum_{j=1}^{\infty}\left(\frac{\tau^2}{2}\right)^j \alpha\left(\frac{x}{2^{\frac{n+j}{s}}}, \frac{x}{2^{\frac{n+j}{s}}},-\frac{x}{2^{\frac{n+j-1}{s}}}\right)
		\\
		&\leq \left(\frac{2}{\tau}\right)^{n-1} \sum_{i=n+1}^{\infty}\left(\frac{\tau^2}{2}\right)^j \alpha\left(\frac{x}{2^{j / s}}, \frac{x}{2^{j/s}},-\frac{x}{2^{j-1/s}}\right)
	\end{align*}
	$\rightarrow 0$ as $n \rightarrow \infty$ (becaure $\dfrac{2}{\tau} \leq  1$ )
	forall $x \in \mathbb{R}$.
	
	Hence $A=A^{\prime}  $.
\end{proof}

\begin{corollary}
	Let $\theta>0$ be a real numbers and $p>s \log _2\left(\dfrac{\tau^2}{2}\right)$. If $\phi:\mathbb{R} \to Y_\rho$ is a mapping such that:
	$$
		\rho\left(\phi(x)+\phi(y)+\phi(z) - q \phi\left(\sqrt[s]{\frac{x^{s} +y^{s}+z^s}{q}}\right)\right)
		\leq  \theta\left(|z|^p+|y|^p+|z|^p\right)
	$$
	for all $x, y, z \in \mathbb{R}$, then there exists a unique radical mapping $A: \mathbb{R} \rightarrow Y_\rho$  such that:
	\begin{equation*}
		\rho(\phi(x)-A(x)) \leq  \frac{\left(2+2^{p / s}\right) \tau^2}{2\left(2^{p / s+1}-\tau^2\right)}|x|^p .
	\end{equation*}
	\end{corollary}
\section{stability of \eqref{Eq-1} in modulon space withont $\Delta_2-$ condition}
The following theorem, presents the stability results of \eqref{Eq-1} in modular space without using the $\Delta_2$-condition.
\begin{theorem}\label{Theorem-2}
	Let $\phi: \mathbb{R} \rightarrow Y_\rho$ be a mapping satifies 
	\begin{equation}\label{Eq-5}
		\rho\left(\phi(x)+\phi(y)+\phi(z)-q \phi\left(\sqrt[s]{\frac{x^s+y^s+z^s}{q}}\right)\right) \leq  \alpha(x, y, z)
	\end{equation}
	and: $\alpha: \mathbb{R}^3 \longrightarrow[0, \infty)$ is a mapping such that:
	\begin{equation}\label{Eq-6}
		\lim_{n \rightarrow \infty} \frac{\alpha\left(2^{n / s} x, 2^{n / s} y, 2^{n / s} z\right)}{2^n}=0 ;\qquad \sum_{j=0}^{\infty} \frac{\alpha\left(2^{j / s} x, 2^{j / s} x,-2^{(j+1)/s} x\right)}{2^j}<\infty
	\end{equation}
	for all $x, y, z \in \mathbb{R}$. 
	Then there exists a unique radical functional mapping $\rho: \mathbb{R} \rightarrow Y_\rho$ such that :
	\begin{equation}\label{Eq-7}
		\rho(\phi(x)-q \phi(0)-A(x)) \leq 
		\frac{1}{2} \sum_{j=0}^{\infty} \frac{1}{2^j} \alpha\left(2^{j / s} x, 2^{j / s}, 2^{\frac{j+1}{s}}\right)
	\end{equation}
	for all
	$x \in \mathbb{R}$.
\end{theorem}
\begin{proof}
	Taking $y=x$ and $z=-2^{1 / s} x$ in \eqref{Eq-5} we get
	\begin{equation*}
		\rho\left(2 \phi(x)-\phi\left(2^{1 / s} x\right)-q \phi(0)\right) \leq  \alpha\left(x, x,-2^{1 / s} x\right)
	\end{equation*}
	Hence: 
	$$\rho\left(\hat{\phi}\left(2^{1 / s} x\right)-2 \hat{\phi}(x)\right) \leq  \alpha\left(x, x,-2^{1 / s} x\right)$$
	where $\hat{\phi}(x)=\phi(x)-q \phi(0)$.
	
	Then we obtain from the convexity of $\rho$ and $\sum_{i=0}^{n-1} \frac{1}{2^{j+1}} \leq  1$
	\begin{align*}
			\rho\left(\hat{\phi}(x)-\frac{\hat{\phi}\left(2^{n / s} x\right)}{2^n}\right)
			&=\rho\left(\sum_{j=0}^{n-1} \frac{1}{2^{j+1}}\left(2 \hat{\phi}\left(2^{j / s} x\right)-\hat{\phi}\left(2^{(j+1)/s} x\right)\right)\right) \\
			& \leq  \sum_{j=0}^{n-1} \frac{1}{2^{j+1}} \rho\left(2 \hat{\phi}\left(2^{{j / s} }x\right)-\hat{\phi}\left(2^{\frac{j+1}{s} }x\right)\right) \\
		&\leq  \frac{1}{2} \sum_{j=0}^{n-1} \frac{1}{2^j} \alpha\left(2^{j / s} x, 2^{j / s} x,-2^{\frac{j+1}{s}}x\right)
	\end{align*}
	for all $x \in \mathbb{R}$ and all positive integers $n>1$.
	
	Then,
	\begin{align*}
			\rho\left(\frac{\hat{\phi}\left(2^{\frac{n}{s} }x\right)}{2^n}-\frac{\hat{\phi}\left(2^{\frac{m}{s} }x\right)}{2^m}\right) &\leq  \frac{1}{2^m} \rho\left(\frac{\hat{\phi}\left(2^{\frac{n-m}{s}} \cdot 2^{\frac{m}{s}}x\right)}{2^{n-m}}-\hat{\phi}\left(2^{\frac{m}{s} }x\right)\right) \\
			& \leq  \frac{1}{2^m} \sum_{j=0}^{n-m-1} \frac{1}{2^{j+1}} \alpha\left(2^{\frac{m+j}{s}}x, 2^{\frac{m+j}{s}} x,-2^{\frac{m+j+1}{s}} x\right) \\
		&
		=\frac{1}{2} \sum_{j=m}^{n-1} \frac{1}{2^j} \alpha\left(2^{j / s} x, 2^{j / s} x,-2^{\frac{j+1}{s}} x\right)
	\end{align*}
	for all $x \in \mathbb{R}$ and all   $m,n\in \mathbb{N}$ with $n>m$. 
	
	Then $\left\{\dfrac{\hat{\phi}\left(2^\frac{n}{s} x\right)}{2^n}\right\}$ is  a $\rho$-Cauchy sequence in $Y_\rho$ wish is $\rho$-complete. Then it's $\rho$-convergent to a mapping
	$A: \mathbb{R} \rightarrow Y_\rho$. Thus we can write:
	$$A(x)=\rho-\lim _{n \rightarrow \infty} \frac{\hat{\phi}\left(2^{\frac{n}{s}}x\right)}{2^n} \text{ i.e. }
	\lim _{n \rightarrow \infty} \rho\left(\frac{\hat{\phi}\left(2^{\frac{n}{s}}x\right)}{2^n}-A(x)\right)=0
	$$
	for all $x \in \mathbb{R}$. By apply the Fatou property, we have:
	\begin{align*}
			\rho(\hat{\phi}(x)-A(x)) & \leq  \liminf _{n \rightarrow \infty} \rho\left(\hat{\phi}(x)-\frac{\hat{\phi}\left(2^{n/s} x\right)}{2^n}\right) \\
			& \leq  \frac{1}{2} \sum_{j=0}^{\infty} \frac{1}{2^j} \alpha\left(2^{j / s} x, 2^{j / s} x,-2^{\frac{j+1}{s}}x\right)
	\end{align*}
	for all $x \in \mathbb{R}$.
	
	In order to prove that $A$ satisfies \eqref{Eq-1},
	we note that 
	\begin{align*}
		&\rho\left(\frac{1}{5} A(x)+\frac{1}{5} A(y)+\frac{1}{5} A(z)-\frac{1}{5} q A\left(\sqrt[s]{\dfrac{x^s+y^s+z^s}{q}}\right)\right)\\
		&\leq\frac{1}{5} \rho\left(A(x)-\frac{\hat{\phi}\left(2^{\frac{n}{s}}x\right)}{2^n}\right)+\frac{1}{5} \rho\left(A(y)-\frac{\hat{\phi}\left(2^{\frac{n}{s} }y\right)}{2^n}\right) \\
		&
		+\frac{1}{5} \rho\left( A(z)-\frac{\hat{\phi}\left(2^{\frac{n}{s}}z \right)} {2^n}\right)
		+ 
		\dfrac{1}{5} \rho\left(q A  \left(\sqrt[s]{\frac{x^s+y^s+z^s}{q}}\right) 
		- 
		 \dfrac{q}{2^n} \hat{\phi}\left( 2^{n/s}\sqrt[s]{\frac{x^s+y^s+z^s}{q}}\right) 
		\right)\\
		& \dfrac{1}{5} \rho\left( \frac{\hat{\phi}\left(2^{\frac{n}{s}}x\right)}{2^n} 
		+
		\frac{\hat{\phi}\left(2^{\frac{n}{s}}y\right)}{2^n}
		+ 
		\frac{\hat{\phi}\left(2^{\frac{n}{s}}z\right)}{2^n} 
		- 
		\dfrac{q}{2^n} \hat{\phi}\left(2^{n/s} \sqrt[s]{\frac{x^s+y^s+z^s}{q}}\right) 
		\right) \\
		& \leq\frac{1}{5} \rho\left(A(x)-\frac{\hat{\phi}\left(2^{\frac{n}{s}}x\right)}{2^n}\right)+\frac{1}{5} \rho\left(A(y)-\frac{\hat{\phi}\left(2^{\frac{n}{s} }y\right)}{2^n}\right) \\
		&
		+\frac{1}{5} \rho\left( A(z)-\frac{\hat{\phi}\left(2^{\frac{n}{s}}z \right)} {2^n}\right)
		+ 
		\dfrac{1}{5} \rho\left(q A  \left(\sqrt[s]{\frac{x^s+y^s+z^s}{q}}\right) 
		- 
		\dfrac{q}{2^n} \hat{\phi}\left( 2^{n/s}\sqrt[s]{\frac{x^s+y^s+z^s}{q}}\right) 
		\right)\\
		& \dfrac{1}{5} \rho\left( \frac{\hat{\phi}\left(2^{\frac{n}{s}}x\right)}{2^n} 
		+
		\frac{\hat{\phi}\left(2^{\frac{n}{s}}y\right)}{2^n}
		+ 
		\frac{\hat{\phi}\left(2^{\frac{n}{s}}z\right)}{2^n} 
		- 
		\dfrac{q}{2^n} \hat{\phi}\left(2^{n/s} \sqrt[s]{\frac{x^s+y^s+z^s}{q}}\right) 
		\right) - (q^2 - 3q) \dfrac{\phi(0)}{2^n}\\
		&+\frac{1}{5} \rho\left( A(z)-\frac{\hat{\phi}\left(2^{\frac{n}{s}}z \right)} {2^n}\right)
		+ 
		\dfrac{1}{5} \rho\left(q A  \left(\sqrt[s]{\frac{x^s+y^s+z^s}{q}}\right) 
		- 
		\dfrac{q}{2^n} \hat{\phi}\left( 2^{n/s}\sqrt[s]{\frac{x^s+y^s+z^s}{q}}\right) 
		\right)\\
		& \dfrac{1}{10} \rho\left( \frac{\hat{\phi}\left(2^{\frac{n}{s}}x\right)}{2^{n-1}} 
		+
		\frac{\hat{\phi}\left(2^{\frac{n}{s}}y\right)}{2^{n-1}}
		+ 
		\frac{\hat{\phi}\left(2^{\frac{n}{s}}z\right)}{2^{n-1}} 
		- 
		\dfrac{q}{2^{n-1}} \hat{\phi}\left(2^{n/s} \sqrt[s]{\frac{x^s+y^s+z^s}{q}}\right) 
		\right) -\dfrac{1}{10}\rho\left(   \dfrac{2q^2 -6q}{2^n}   {\phi(0)} \right) \\
		&\leq\frac{1}{5} \rho\left(A(x)-\frac{\hat{\phi}\left(2^{\frac{n}{5}}x\right)}{2^n}\right)+\frac{1}{5} \rho\left(A(y)-\frac{\hat{\phi}\left(2^{\frac{n}{s} }y\right)}{2^n}\right) \\
		&+\frac{1}{5} \rho\left( A(z)-\frac{\hat{\phi}\left(2^{\frac{n}{s}}z \right)} {2^n}\right)
		+ 
		\dfrac{1}{5} \rho\left(q A  \left(\sqrt[s]{\frac{x^s+y^s+z^s}{q}}\right) 
		- 
		\dfrac{q}{2^n} \hat{\phi}\left( 2^{n/s}\sqrt[s]{\frac{x^s+y^s+z^s}{q}}\right) 
		\right)\\
		& -\dfrac{1}{5}\dfrac{\alpha( 2^{n/s}x,2^{n/s}y,2^{n/s}z)}{2^n} +\dfrac{1}{10}\rho\left(   \dfrac{2q^2 -6q}{2^n}   {\phi(0)} \right)
		\rightarrow 0 \text { as } n \rightarrow \infty \text { (because } \phi(0) \in Y_p
	\end{align*}
	and $\lim _{n \rightarrow \infty} \frac{2 q^2-6 q}{2^n}=0$.
	
	Then $$A(x)+A(y)+A(z) = q A  \left(\sqrt[s]{\frac{x^s+y^s+z^s}{q}}\right)  $$
	
	Now, we prove that $A$ is unique Let $A^{\prime}$ be an other radical mapping satisfying the estimation \eqref{Eq-7}.
	
	We have 
	\begin{align*}
		\rho\left(\frac{1}{2} A(x)-\frac{1}{2} A^{\prime}(x)\right)&\leq 
		\frac{1}{2} \rho\left(\frac{A\left(2^{\frac{n}{s}}x\right)}{2^n}-\frac{\hat{\phi}\left(2^ \frac{n}{s}x\right)}{2^n}\right)+\frac{1}{2} \rho\left(\frac{\hat{\phi}\left(2^{\frac{n}{s}}x\right)}{2^n}-\frac{A^{'}\left(2^{\frac{n}{s}} x\right)}{2^n}\right)\\
		& \leq \frac{1}{2^{n+1}} \rho\left(A\left(2^{\frac{n}{s}}x\right)-\hat{\phi}\left(2^{\frac{n}{s}}x\right)\right)+\frac{1}{2^{n+1}} \rho\left(\hat{\phi}\left(2^{\frac{n}{s}}x\right)-A^{'}\left(2^{\frac{n}{s}  }x\right)\right)\\
		&\leq \frac{1}{2^{n}}\sum_{j=0}^{\infty} \frac{1}{2^{j+1}} \alpha\left(2^{\frac{n+j}{s}} x, 2^{\frac{n+j}{s}} x,-2^{\frac{n+j+1}{s}} x\right)\\
		&\leq  \sum_{j=0}^{\infty} \frac{1}{2^{j+1}} \alpha\left(2^{\frac{ j}{s}} x, 2^{\frac{ j}{s}} x,-2^{\frac{ j+1}{s}} x\right) \rightarrow 0 \text { as } n \rightarrow \infty 
	\end{align*}
	Then $A=A^{\prime}$, this complete the proof
\end{proof}
\section{Stability of \eqref{Eq-1} in modular space using the point five theorem}
The following theorem of fixed point in modular space will play in important role in establishing our stability remits of \eqref{Eq-1}.
\begin{theorem}[\cite{Khamsi2008}]
	Let $Y_\rho$ be a modular space and $C$ be a $p$-complete nonempty subset of $Y_\rho$. A mapping $T: C \rightarrow Y_\rho$ is called quasi-contration if and only if there exits $K<1$ such that:
	\begin{equation*}
			\rho(T(x)-T(y)) \leq   K \max \{\rho(x-y), \rho(x-T(x)), \rho(y-T(y)), 
			\rho(x-T(y)), \rho(y-T(x))\} .
	\end{equation*}
	Let $x \in C$ such that: $\delta_p(x)=\sup\left\{\rho\left(T^n(x)-T^m(x)\right) ; m,n \in \mathbb{N}\right\}$.
	
	Then $\left\{T^n(x)\right\}$ is $\rho$-convergent to a point $w \in C$. Moreover, if $\rho(\omega-T(\omega))<\infty$ and $\rho(x-T(\omega))<\infty$, Then the $\rho$-limit of $T(x)^n$ is a fixed point of $T$ and   if $\bar{\omega}$ is any fixed point of $T$ in $C$ such that $\rho(\omega-\bar{\omega})<\infty$, then $\omega=\bar{\omega}$.
\end{theorem}

\begin{theorem}
	Let $Y_\rho$ be a $\rho$-complete modular space satisfying $\Delta_2$-condition and $\phi: \mathbb{R} \rightarrow Y_\rho$ be a mapping such that: 
	\begin{equation}\label{Eq-8}
		\rho\left(\phi(x)+\phi(y)+\phi(z)-q \phi\left(\sqrt[s]{\dfrac{x^s+y^s+z^s}{q}}\right)\right) \leq  \alpha(x, y, z)
	\end{equation}
	 for all $x, y, z \in \mathbb{R}$, where $\alpha: \mathbb{R}^3 \rightarrow[0,\infty)$ is a mapping  such that: $$\alpha\left(2^{1 / s} x, 2^{1 / s} x, -2^{2 / s} x\right) \leq  2 L \alpha(x,x,-2^{1 / s}x) \qquad (0<L<1)$$ and 
	 \begin{equation}
	 	\lim_{n \rightarrow \infty} \frac{\alpha\left(2^{n / s} x, 2^{n /s} y, 2^{n / s} z\right)}{2^n}=0..
	 \end{equation}
	Then there exists a unique radical mapping $\hat{\phi}: \mathbb{R} \rightarrow Y_\rho$ such that $\rho(\phi(x)-\hat{\phi}(x)) \leq \dfrac{1}{2(1-L)} \alpha(x,x,-2^{1/s} x)$ for all $x \in \mathbb{R}$.
\end{theorem}

\begin{proof}
We define the set: $N=\{g : \mathbb{R} \rightarrow Y_p\}$ and we introduce the mapping $\hat{\rho}$ on $N$ as :
	\begin{equation*}
			  \hat{\rho}(g)=\inf \{\lambda>0 : \rho(g(x)) \leq  \lambda\alpha(x,x,-2^{1/s} x)\}
	\end{equation*}
	by same methode as in the proof of {Theorem 2.1 in \cite{Sadeghi}}, we deduce that $\hat{\rho}$ is convex modular on $N$, and $N$ is $\hat{\rho}$-complete.
	
	Now, we introduce the mapping $A: N_{\hat{\rho}} \rightarrow N_{\hat{\rho}}$ as:
	\begin{equation*}
		\Lambda(g)(x)=\dfrac{1}{2} g\left(2^{1/s} x\right) ; \qquad
		g \in N \text { and } x \in \mathbb{R} .
	\end{equation*}
	
	Let $f, g \in N_{\hat{\rho}}$ and $c$ be positive real number with: $\hat{\rho}(f-g) \leq  c$, we have:
	$\rho(f(x)-g(x)) \leq  c\alpha(x, x,-2^{1 / s} x)$ for all $x \in \mathbb{R}$.
	
	Then
	\begin{align*}
			\rho\left(\frac{f\left(2^{1 /s}x\right)}{2}-\frac{g\left(2^{1 / s} x\right)}{2}\right) \leq  \dfrac{1 }{2} \rho\left(f\left(2^{1 / s} x\right)-g\left(2^{1 / s} x\right)\right)\\
			\leq  \dfrac{1 }{2} c\alpha\left( \left(2^{1 / s} x\right),\left(2^{1 / s} x\right),\left(-2^{2 / s} x\right) \right) \\
			\leq Lc\alpha(x, x,-2^{1 / s} x)
	\end{align*}
	for all $x\in\mathbb{R}$.
 
	Hence $$\hat{\rho}(\Lambda f-\Lambda g) \leq  L \hat{\rho}(f-g)$$ for all $f, g \in N_{ \hat{\rho}}$. 
	
	So $A$ is $\hat{\rho}$-strict contraction.
	
	Letting $y=x$ and $z=-2^{1 / s} x$ in \eqref{Eq-8}, we get:
	\begin{equation*}
		\rho\left(2 \phi(x)-\phi\left(2^{1 / s} x\right)\right) \leq  \alpha\left(x, x,-2^{1 / s} x\right)
	\end{equation*}
	Then.
	
	$\begin{aligned} &   \rho\left(\frac{1}{2^2} \phi\left(2^{2 / s} x\right)-\phi(x)\right) \leq  1 / 2 \rho\left(1 / 2 \phi\left(2^{2 / s} x\right)-\phi\left(2^{1 / s} x\right)\right)   +1 / 2 \rho\left(\phi\left(2^{1 / s}\right)-2 \phi(x)\right) \\ & \leq  \frac{1}{2^2} \alpha\left(2^{1 / s} x, 2^{1 / s} x,-2^{2 / s} x\right)+\frac{1}{2} \alpha\left(x, x,-2^{1 / s} x\right) \\ & \end{aligned}$
	
	by mathematical induction we have:
	\begin{align*}
		\rho\left(\frac{1}{2^n} \phi\left(2^{n / s} x\right)-\phi(x) \right)&\leq
		\sum_{i=1}^{n} \frac{1}{2^i} \alpha\left(2^{\frac{i-1}{s}} x, 2^{\frac{i-1}{s}} x,-2^{\frac{i}{s}} x\right)\\
		& = \sum_{i=0}^{n-1} \frac{1}{2^{i+1}} \alpha\left(2^{\frac{i}{s}}, 2^{\frac{i}{s}} x,-2^{\frac{i+1}{s}} x\right)\\
		&\leq\sum_{i=0}^{n-1} \frac{1}{2^{i+1}}(2 L)^i \alpha\left(x, x,-2^{1 / s} x\right)\\
		&\leq  \frac{1}{2(1-L)} \alpha\left(x, x,-2^{1 / s} x\right)
	\end{align*}
	Then
	\begin{equation*}
		\begin{aligned}
			\rho\left(\frac{\phi\left(2^{n/s}x\right)}{2^n}-\frac{\phi\left(2^{\frac{m}{s}}x\right)}{2^m}\right) & \leq  \frac{1}{2} \rho\left(2 \frac{\phi\left(2^{\frac{n}{s}}x\right)}{2^n}-2 \phi(x)\right)+\frac{1}{2} \rho\left(2\frac{\phi\left(2^{\frac{m}{s}}\right)}{2^m}-2 \phi(x)\right) \\
			& \leq  \frac{\tau}{2} \rho\left(\frac{\phi\left(2^{\frac{n}{s} }x\right)}{2^n}-\phi(x)\right)+\frac{\tau}{2} \rho\left(\frac{\phi\left(2^{\frac{m}{s}}x\right)}{2^m}-\phi(x)\right) \\
			& \leq  \frac{\tau}{2(1-L)} \alpha\left(x, x,-2^{1 / s} x\right)
		\end{aligned}
	\end{equation*}
	
	Hence $\hat{\rho}\left(\Lambda^n \phi-\Lambda^m \phi\right) \leq  \dfrac{\tau}{2(1-L)}$.
	
	Thus $\delta_{\hat{p}}(\phi)<0$ and $\left\{\Lambda^n \phi\right\}$ is $\hat{\rho}$-convergent to $\hat{\phi} \in N_{\hat{\rho}}$.
	
	By $\hat{\rho}$-contraction of $\Lambda$ we have:
	\begin{equation*}
		\hat{\rho}\left(\Lambda^n \phi-\Lambda \hat{\phi}\right) \leq  L \hat{\rho}\left(\Lambda^{n-1} \phi-\hat{\phi}\right)
	\end{equation*}
	
	Letting $n \rightarrow \infty$ and apply $\hat{\rho}$-atom property:
	\begin{align*}
			\hat{\rho}(\Lambda \hat{\phi}-\hat{\phi}) & \leq  \liminf _{n \rightarrow \infty}   \hat{\rho}\left(\Lambda \hat{\phi}-\Lambda^{n} \phi\right) \\
			& \leq  L \liminf_{n \rightarrow \infty} \hat{\rho}\left(\hat{\phi}-\Lambda^{n-1} \phi\right)=0.
	\end{align*}
	Then, we conclude that $\hat{\phi}$ is a fixed puint of $\Lambda$.
	Then, we conclude that $\phi$ is a fxed poin of $A$.
	By same methode as in the proof of Theorem \ref{Theorem-1}, we deduce that $\hat{\phi}$ satifies \eqref{Eq-1}.
	
	Now, we must prove that $\hat{\phi}$ is unique.
	
	Let $\bar{\phi}$ be another fixed point of $\Lambda$.
	
	We have
	\begin{align*}
			\hat{\rho}(\bar{\phi}-\hat{\phi}) & =\hat{\rho}(\Lambda \bar{\phi}-\Lambda \hat{\phi}) \\
			& \leq  L \hat{\rho}(\bar{\phi}-\hat{\phi})
	\end{align*}
	
	Hence $\bar{\phi}=\hat{\phi}$
\end{proof}
 
 \section*{Declarations}
 
 \medskip
 
 \noindent \textbf{Availablity of data and materials}\newline
 \noindent Not applicable.
 
 \medskip

 \noindent \textbf{Competing  interest}\newline
 \noindent The authors declare that they have no competing interests.

 \medskip
 
 \noindent \textbf{Fundings}\newline
 \noindent  Authors declare that there is no funding available for this article.

 \medskip
 
 \noindent \textbf{Authors' contributions}\newline
 \noindent The authors equally conceived of the study, participated in its
 design and coordination, drafted the manuscript, participated in the
 sequence alignment, and read and approved the final manuscript.


\medskip
 
\begin{thebibliography}{99}

\bibitem{Amemiya}  I. Amemiya, On the representation of complemented modular lattices. J. Math. Soc. Jpn. 1957, 9, 263-279. 

\bibitem{Aoki} T. Aoki,  On the stability of the linear transformation in Banach spaces. J. Math. Soc. Jpn. 1950, 2, 64-66. 
	
\bibitem{Găvruţa} P. Găvruţa,  A generalization of the Hyers-Ulam-Rassias stability of approximately additive mappings. J. Math. Anal. Appl. 1994, 184, 431-436. 

\bibitem{Hyers}  D. H. Hyers,  On the stability of the linear functional equation. Proc. Nat. Acad. Sci. USA 1941, 27, 222-224.  
	
\bibitem{Khamsi2008} M. A. Khamsi, 
	Quasicontraction mapping in 
	modular spaces 
	without $\Delta_2-$condition,
	Fixed Point
	Theory Appl., (2008), Art. ID 916187.
	
\bibitem{Koshi}  S. Koshi, T. Shimogaki,  On F-norms of quasi-modular spaces. J. Fac. Sci. Hokkaido Univ. Ser. I Math. 1961, 15, 202-218. 
	
\bibitem{Krbec} M.  Krbec, Modular interpolation spaces I. Z. Anal. Anwendungen 1982, 1, 25-40. 

\bibitem{Maligranda} L. Maligranda,  Orlicz Spaces and Interpolation; Seminários de Matemática, 5; Universidade Estadual de Campinas, Departamento de Matemática: Campinas, Brazil, 1989.

\bibitem{Musielak} J.  Musielak, Orlicz Spaces and Modular Spaces; Lecture Notes in Mathematics, 1034; Springer: Berlin/Heidelberg, Germany, 1983.
	
\bibitem{Nakano} H. Nakano,  Modulared Semi-Ordered Linear Spaces; Maruzen Co., Ltd.: Tokyo, Japan, 1950.

\bibitem{Orlicz}  W. Orlicz,  Collected Papers. Part I, II; PWN—Polish Scientific Publishers: Warsaw, Poland, 1988.

\bibitem{Rassias} T. M. Rassias,  On the stability of the linear mapping in Banach spaces. Proc. Am. Math. Soc. 1978, 72, 297-300. 

\bibitem{Sadeghi}  Gh. Sadeghi, A fixed point approach to stability of functional equations in modular spaces, Bull.
Malays. Math. Sci. Soc. 37 (2) (2014), 333-344.

\bibitem{Ulam} S. M. Ulam, A Collection of Mathematical Problems, Interscience Tracts in Pure and Applied Mathematics, No. 8; Interscience Publishers: New York, NY, USA, 1960.
	
\bibitem{Yamamuro} S.  Yamamuro, On conjugate spaces of Nakano spaces. Trans. Am. Math. Soc. 1959, 90, 291-311. 
		
\end{thebibliography}
\end{document}